\documentclass[12pt, a4paper, leqno]{amsart}

\usepackage{amsmath}
\usepackage{amsfonts}
\usepackage{amssymb}
\usepackage{txfonts}
\usepackage{comment}
\usepackage{ae}
\usepackage{graphicx}
 \usepackage[colorlinks]{hyperref}

\newcommand{\R}{\varmathbb{R}}
\newcommand{\Z}{\varmathbb{Z}}
\newcommand{\N}{\varmathbb{N}}
\newcommand{\Rn}{\varmathbb{R}^n}

\newcommand{\C}{\mathcal{C}}

\def\b{\qopname\relax o{b}}
\def\diam{\qopname\relax o{diam}}

\def\dist{\qopname\relax o{dist}}
\def\min{\qopname\relax o{min}}
\def\diam{\qopname\relax o{diam}}

\usepackage{amsthm}

\swapnumbers
\theoremstyle{plain}
\newtheorem{theorem}[equation]{Theorem}
\newtheorem{lemma}[equation]{Lemma}
\newtheorem{prop}[equation]{Proposition}
\newtheorem{corollary}[equation]{Corollary}

\theoremstyle{definition}
\newtheorem{definition}[equation]{Definition}

\theoremstyle{remark}
\newtheorem{remark}[equation]{Remark}

\numberwithin{equation}{section}

\title{On fractional Poincar\'e inequalities}
\author{Ritva Hurri-Syrj\"anen and Antti V. V\"ah\"akangas}
\address{Department of Mathematics and Statistics, 
Gustaf H\"allstr\"omin katu 2$\b$, FI-00014 University of Helsinki, Finland}
\email{ritva.hurri-syrjanen@helsinki.fi}
\email{antti.vahakangas@helsinki.fi}
\thanks{A.V.V.
was supported by the Academy of Finland, grants 75166001 and 1134757.
}
\date{\today}

\begin{document}

\keywords{fractional Poincar\'e inequality, fractional Sobolev--Poincar\'e inequality, uniform domain, $s$-John domain, porous set}
\subjclass[2010]{46E35 (26D10)}

\begin{abstract}
We show that fractional $(p,p)$-Poincar\'e inequalities and even fractional Sobolev--Poincar\'e  inequalities hold for
bounded John domains, and especially for bounded Lipschitz domains. We also prove
sharp fractional $(1,p)$-Poincar\'e inequalities for $s$-John domains.
\end{abstract}

\maketitle

\markboth{\textsc{R. Hurri-Syrj\"anen and A. V. V\"ah\"akangas}}
{\textsc{On fractional Poincar\'e inequalities}}

\section{Introduction}

We consider the following fractional $(q,p)$-Poincar\'e inequality  in a bounded domain $G$ in $\Rn\,,$
$n\geq 2\,,$ 
\begin{equation}\label{fractionalqp}
\int_G\vert u(x)-u_G\vert ^q\,dx
\le
c
\biggl(\int_G\int_{G\cap B^n(x,\tau \dist(x,\partial G)
)}\frac{\vert u(x)-u(y)\vert ^p}{\vert x-y\vert ^{n+\delta p}}\,dy\,dx
\biggr)^{q/p}\,,
\end{equation}
where $1\le p,q <\infty$,
$\delta,\tau \in (0,1)$, and
the constant $c$ does not depend on $u\in L^p(G)$.
Our inequality \eqref{fractionalqp} with $q=p$ is stronger than 
the fractional inequality
\begin{equation}\label{fractionalpp}
\int_G\vert u(x)-u_G\vert ^p\,dx
\le
c
\int_G\int_G\frac{\vert u(x)-u(y)\vert ^p}{\vert x-y\vert ^{n+\delta p}}\,dx\,dy\,,
\end{equation}
where on the right hand side is the commonly used seminorm  on $W^{\delta,p}(G)$, \cite{A}.
Augusto C. Ponce showed that bounded Lipschitz domains support
the same type of inequalities 
as \eqref{fractionalpp} but with general radial weights, \cite{P1}, \cite[Theorem 1.1]{P2}.
Jean Bourgain, Ha\"im Brezis, and Petru Mironescu 
found the optimal constant $c$
in \eqref{fractionalpp} when $G$ is a cube \cite[Theorem 1]{BBM2}. 
An elementary proof was provided by Vladimir Maz'ya and Tatyana Shaposhnikova, \cite{MS1},
\cite{MS2}.
The relationship between the right hand side of \eqref{fractionalpp} and the $L^p(G)$ integrability
of the absolute value of the gradient in smooth bounded domains is considered 
in \cite{BBM1}.

We give sufficient geometric conditions for a bounded domain $G$ in $\R^n$ to
support the fractional $(q,p)$-Poincar\'e inequality for $1\le q\le p<\infty$, Theorem \ref{thmP}. Examples of the domains which support the
fractional $(p,p)$-Poincar\'e inequality are
John domains, Theorem \ref{pp_John}. The John domains
include uniform domains and hence also Lipschitz domains.
We show that John domains
support the fractional Poincar\'e inequality
\eqref{fractionalqp} when $1<p\le q\le np/(n-\delta p)\,$ and $p<n/\delta\,,$ Theorem 
\ref{fractional_sobolev_poincare}.
We also study more general bounded domains, 
so called $s$-John domains with $s>1$. We prove
fractional $(1,p)$-Poincar\'e inequalities for
these domains, Theorem \ref{sharp}, and we show that these results are sharp, Theorem \ref{1p_counter}.

\section{Notation and auxiliary results}

We assume that $G$ is a bounded domain in Euclidean $n$-space
$\Rn$, $n\geq 2$, throughout the paper. 

We denote by
$\mathcal{D}$ the family of closed dyadic cubes in $\R^n$. We let $\mathcal{D}_j$ be
the family of those dyadic cubes whose
side length is $2^{-j}$, $j\in\Z$.
For a domain $G$ we fix its Whitney decomposition $W=W_G\subset\mathcal{D}$. 
For the properties of dyadic cubes and Whitney cubes we refer to Elias M. Stein's book,
\cite{S}.  We write $Q^*=\frac{9}{8}Q$ for $Q\in W$. Then,
\begin{equation}\label{dist_est}
\frac{3}{4}\diam(Q)\le \dist(x,\partial G)\le 6 \diam(Q)\,,\quad \text{ if }x\in Q^*.
\end{equation}

Let us fix a cube $Q_0$ in the Whitney decomposition $W$.
For each $Q\in W$
there exists a chain of cubes 
$(Q_0^*,Q_1^*,\cdots \,,Q_k^*)=:\C(Q^*)$
joining
two cubes $Q_0^*$ and $Q_k^*=Q^*$
such that
$
Q_i^*\cap Q_j^*\neq\emptyset\,
$
if and only if $\vert i-j\vert\le 1$.
The length of  this chain is written as
$\ell(\C(Q^*)):=k$.
Once the chains of cubes have been
picked up, then for each Whitney cube $A$ we define  a set
$A(W)=\{Q\in W\mid A^*\in \C(Q^*)\}$.
We call this construction a chain decomposition of $G$
with a fixed cube $Q_0$.

The side length of a cube $Q$ in $\R^n$ is denoted by $\ell (Q)$.
We write $\chi_E$ for the characteristic function of a set $E$.
The Lebesque $n$-measure of a  measurable set $E$ is denoted by $\vert E\vert.$
The upper Minkowski dimension of a set $E$ in $\R^n$ is
\[
\dim_\mathcal{M}(E):=\sup \big\{\lambda\ge 0\mid \limsup_{r\to 0+} \mathcal{M}_\lambda(E,r)=\infty\big\},
\]
where
\[
\mathcal{M}_\lambda(E,r):=\frac{|E+B^n(0,r)|}{r^{n-\lambda}}=\frac{|\cup_{x\in E}B^n(x,r)|}{r^{n-\lambda}},\quad r> 0,
\]
is the $\lambda$-dimensional Minkowski precontent.

The notation $a\lesssim b$ is used to express that an estimate $a\le cb$ holds for some constant $c>0$
whose value is clear from the context. 
We use subscripts to indicate the dependence on parameters, for example,
a quantity $c_{\lambda}$ depends on a parameter $\lambda$.

The following lemma gives a fractional inequality in a cube.

\begin{lemma}\label{inequality_cube}
Let $Q$ be a closed cube in $\Rn$. Let $1\le q\le p <\infty $
and let $\delta,\rho \in (0,1)\,.$
Then, there is a constant $c<\infty $ independent of $u\in L^p(Q)$ such that
\begin{equation*}
\begin{split}
&\frac{1}{|Q|} \int_{Q} |u(y)-u_{Q}|^q\,dy 
\\& \le c|Q|^{q(\delta /n-1/p)}\bigg(\int_{Q} \int_{Q\cap B^n(y,\,\rho \ell(Q))} \frac{|u(y)-u(z)|^p}{|z-y|^{n+\delta p}}\,dz\,dy\bigg)^{q/p}\,.
\end{split}
\end{equation*}
\end{lemma}

\begin{proof}
Without loss of generality we may assume that $Q=[0,1]^n$. 
This comes from a simple scaling and translation argument.

Let us divide $Q$ into $k^n$ congruent and closed subcubes $Q_1,\ldots,Q_{k^n}$, where $k$ is chosen such that $R\subset B^n(y,\rho)$ for every $y\in R$ whenever
$R$ is a union of two cubes $Q_i$ and $Q_j$, $i,j\in \{1,2,\ldots,k^n\}$, sharing
a common face; in particular, the case $i=j$ is allowed.
We obtain
\begin{equation}\label{rest}
\begin{split}
\frac{1}{|R|} \int_{R} |u(y)-u_{R}|^q\,dy 
&\le \bigg(\frac{1}{|R|} \int_{R} |u(y)-u_{R}|^p\,dy\bigg)^{q/p}\\
&\le \bigg(\frac{1}{|R|} \int_{R} \frac{1}{|R|} \int_{R} |u(y)-u(z)|^p\,dz\,dy\bigg)^{q/p}\\
&\lesssim |R|^{q(\delta/n-1/p)}\bigg(\int_{R} \int_{R} \frac{|u(y)-u(z)|^p}{|z-y|^{n+\delta p}}\,dz\,dy\bigg)^{q/p}\\
&\lesssim \bigg(\int_{Q} \int_{Q\cap B^n(y,\,\rho)} \frac{|u(y)-u(z)|^p}{|z-y|^{n+\delta p}}\,dz\,dy\bigg)^{q/p}\,.
\end{split}
\end{equation}
H\"older's inequality and Minkowski's inequality yield
\begin{align}\label{kaksi}
&\frac{1}{|Q|} \int_Q |u(y)-u_Q|^qdy \lesssim \frac{1}{|Q|} \int_Q |u(y)-u_{Q_1}|^qdy \notag\\
&\lesssim \sum_{j=1}^{k^n} \int_{Q_j}|u(y)-u_{Q_j}|^qdy + \sum_{j=1}^{k^n} \int_{Q_j} |u_{Q_j}-u_{Q_1}|^q dy\,.
\end{align}
By \eqref{rest} it is enough to estimate the second series in \eqref{kaksi}. Let us
fix $Q_j$, $j\in \{1,\ldots,k^n\}$, and let
$\sigma:\{1,2,\ldots,kn\}\to \{1,2,\ldots,k^n\}$ be 
such that
$\sigma(1)=1$, $\sigma(kn)=j$, and the subsequent cubes
 $Q_{\sigma(i)}$ and $Q_{\sigma(i+1)}$  share
a common face if $i=1,\ldots,kn-1$.
Since $kn\lesssim 1$, we obtain
\begin{align}\label{sums}
|u_{Q_{j}}-u_{Q_1}|^q &\le \bigg(\sum_{i=1}^{kn-1} |u_{Q_{\sigma(i+1)}}-u_{Q_{\sigma(i)}}|\bigg)^q\notag\\
&\lesssim \sum_{i=1}^{kn-1} |u_{Q_{\sigma(i+1)}}-u_{Q_{\sigma(i+1)}\cup Q_{\sigma(i)}}|^q+\sum_{i=1}^{kn-1}|u_{Q_{\sigma(i+1)}\cup Q_{\sigma(i)}}-u_{Q_{\sigma(i)}}|^q.
\end{align}
Let us consider the first sum in \eqref{sums}. Note that
\begin{align*}
&|u_{Q_{\sigma(i+1)}}-u_{Q_{\sigma(i+1)}\cup Q_{\sigma(i)}}|^q\\&\le \frac{1}{|Q_{\sigma(i+1)}|}\int_{Q_{\sigma(i+1)}} |u_{Q_{\sigma(i+1)}} -u(y) + u(y)-u_{Q_{\sigma(i+1)}\cup Q_{\sigma(i)}}|^q dy\\
&\lesssim \frac{1}{|Q_{\sigma(i+1)}|}\int_{Q_{\sigma(i+1)}} |u(y)-u_{Q_{\sigma(i+1)}}|^qdy\\
&\qquad + \frac{1}{|{Q_{\sigma(i+1)}\cup Q_{\sigma(i)}}|} \int_{{Q_{\sigma(i+1)}\cup Q_{\sigma(i)}}} |u(y)-u_{{Q_{\sigma(i+1)}\cup Q_{\sigma(i)}}}|^q dy.
\end{align*}
By \eqref{rest} we obtain
\[
\sum_{i=1}^{kn-1} |u_{Q_{\sigma(i+1)}}-u_{Q_{\sigma(i+1)}\cup Q_{\sigma(i)}}|^q
\lesssim \bigg(\int_{Q} \int_{Q\cap B^n(y,\,\rho)} \frac{|u(y)-u(z)|^p}{|z-y|^{n+\delta p}}\,dz\,dy\bigg)^{q/p}.
\]
Similar estimates for the remaining sum in \eqref{sums}  conclude the proof.
\end{proof}

We also need some estimates involving porous sets in $\R^n$.

\begin{definition}\label{porous}
A  set $S$ in Euclidean $n$-space 
is {\em porous in $\R^n$}  if for some $\kappa\in (0,1]$
the following statement is true:
for every  $x\in\R^n$ and $0<r\le 1$ there is $y\in B^n(x,r)$ such that
$B^n(y,\kappa\,r)\cap S=\emptyset$.
\end{definition}

The  following lemma gives a norm estimate related to
porous sets, and it is based on maximal function techniques. This lemma might be of independent interest.

\begin{lemma}\label{por_lem}
Suppose that $S$ is porous  in $\R^n$ and let $1\le p<\infty$.
If 
$x\in S$ and $0<r\le 1$, then
\[
\int_{B^n(x,r)} \log^p\frac{1}{\dist(y,S)}\,dy
\le  cr^n(1+\log^p r^{-1}),
\]
where the constant $c$ is independent of $x$ and $r$.
\end{lemma}

\begin{proof}
Let us write
\[
\mathcal{C}_{S} = \{R\in\mathcal{D}\,:\,\mathrm{dist}(x_R,S)/(4+\sqrt n)\le \ell(R)\le 1\},
\]
where $x_R$ is the midpoint of a dyadic cube $R$.
Suppose that
$R\in\mathcal{D}$ is such that $\ell(R)\le 1$ and
$\dist(y,S)\le 4\ell(R)$ for some $y\in R$. Then, since
\begin{equation}\label{r_inclusion}
\begin{split}
\dist(x_R,S)&\le \dist(x_R,y) + \dist(y,S)\\&\le \sqrt n\ell(R) + \dist(y,S) \le (4+\sqrt n)\ell(R)
\end{split}
\end{equation}
for the midpoint of $R$, we conclude that $R\in\mathcal{C}_S$.

Fix $j\in \N_0$ such that $2^{-j}\le r< 2^{-j+1}$, and
consider a dyadic cube $Q\in\mathcal{D}_{j}$ for which
$Q\cap B^n(x,r)\not=\emptyset$. 
By covering $B^n(x,r)$ with
such dyadic cubes it is enough to show that
\begin{equation}\label{test}
|| \log\mathrm{dist}(\cdot,S)^{-1} ||_{L^p(Q\cap B^n(x,r))}^p
\lesssim  r^n(1+\log^p r^{-1}).
\end{equation}
By the porosity and the Lebesgue density theorem, the  $n$-measure
of $S$ is zero. 
Hence, it is enough to consider points $y\in Q\cap B^n(x,r)\setminus S$. Since $x\in S$,
\begin{equation}\label{q_est}
1\le \frac{2\ell(Q)}{\dist(y,S)}.
\end{equation}
Let us consider a finite sequence of dyadic cubes 
\[
Q=Q_0(y)\supset Q_1(y)\supset \dotsb \supset Q_m(y),
\]
each of them containing the point $y$ and satisfying 
\begin{equation}\label{ratio}
\ell(Q_i(y))/\ell(Q_{i+1}(y))=2,\qquad i=0,1,\ldots,m-1.
\end{equation} The 
last cube is chosen to satisfy
\begin{equation}\label{qsat}
\dist(y,S)/4 \le \ell(Q_m(y)) < \dist(y,S)/2.
\end{equation}
From \eqref{q_est} it follows that $m\ge 1$.
By \eqref{ratio} and \eqref{q_est}
\[
2^m = \prod_{i=0}^{m-1} \frac{\ell(Q_i(y))}{\ell(Q_{i+1}(y))} =\frac{\ell(Q_0(y))}{\ell(Q_m(y))}
>\frac{2\ell(Q_0(y))}{\dist(y,S)}=\frac{2\ell(Q)}{\dist(y,S)}\ge 1.
\]
Hence,
\[
m\ge \log 2^m\ge  \log 2\ell(Q) - \log \dist(y,S)\ge 0.
\]
Furthermore, \eqref{qsat}  and \eqref{r_inclusion} yield
$Q_i(y)\in \mathcal{C}_S$ if $i=0,1,\ldots,m$.
Thus, we obtain
\begin{align*}
\sum_{\substack{R\in\mathcal{C}_S\\R\subset Q}} \chi_R(y) \ge 1+m &\ge 1+\log (\ell(Q)) - \log \dist(y,S)
\ge 0,
\end{align*}
where $\chi_R$ is the characteristic function of $R$.
Integrating this inequality and using triangle-inequality yields
\begin{align*}
&||\log \dist(\cdot,S)^{-1} ||_{L^p(Q\cap B^n(x,r))}
\\&\le |1+\log \ell(Q)|\,|Q\cap B^n(x,r)|^{1/p} +\bigg\|  \sum_{\substack{R\in\mathcal{C}_S\\R\subset Q}} \chi_R \bigg\|_{L^p(\R^n)}.
\end{align*}
Since $S$ is porous in $\R^n$, we may 
follow the proof of 
\cite[Theorem 2.10]{ihna}. We obtain a constant
$K_\kappa$, depending on $\kappa$ in Definition \ref{porous}, and
families 
\[\{\hat R\}_{R\in\mathcal{C}_S^k},\quad \mathcal{C}_S^k\subset \mathcal{C}_S,\qquad k=0,1,\ldots,K_\kappa-1,\]
where each $\{\hat R\}_{R\in\mathcal{C}_S^k}$ is a disjoint family of cubes $\hat R\subset R$, such that
\begin{align*}
 \bigg\|  \sum_{\substack{R\in\mathcal{C}_S\\R\subset Q}} \chi_R \bigg\|_{L^p(\R^n)}
\lesssim \sum_{k=0}^{K_\kappa-1} \bigg\|  \sum_{\substack{R\in\mathcal{C}_S^k\\R\subset Q}} \chi_{\hat R} \bigg\|_{L^p(\R^n)}
\le \sum_{k=0}^{K_\kappa-1} ||\chi_Q||_{L^p(\R^n)}\lesssim |Q|^{1/p}.
\end{align*}
By combining the estimates we obtain
\begin{align*}
||\log \dist(\cdot,S)^{-1} ||_{L^p(Q\cap B^n(x,r))} 
\lesssim (1+\log\ell(Q)^{-1})|Q|^{1/p}
\lesssim (1+\log r^{-1})r^{n/p}.
\end{align*}
Estimate \eqref{test} follows.
\end{proof}

\section{Conditions for the fractional Poincar\'e inequality}

In the following theorem we give sufficient conditions for
a bounded domain to support
the fractional $(q,p)$-Poincar\'e inequality \eqref{fractionalqp}.

\begin{theorem}\label{thmP}
Let $G$ be a bounded domain in $n$-dimensional Euclidean space, $n\geq 2\,,$ 
with a Whitney decomposition $W$.
Let $1\le q\le p<\infty\,$ and let $\delta,\tau\in (0,1)$.

(1) If $q<p$ and if there exists a chain decomposition of $G$ such that
\begin{equation}\label{sharpe}
\sum_{A\in W} \Bigg(\sum_{Q\in A(W)}
\ell(\mathcal{C}(Q^*))^{q-1}|Q|\,  |A|^{q(\delta /n-1/p)}\Bigg)^{p/(p-q)}<\infty\,,
\end{equation}
then  $G$ supports the fractional $(q,p)$-Poincar\'e inequality 
\eqref{fractionalqp}.

(2) If $q=p$ and if there exists a chain decomposition of $G$ such that
\begin{equation}\label{pp}
\sup_{A\in W}
\sum_{Q\in A(W)}
\ell(\mathcal{C}(Q^*))^{p-1}|Q|\,  |A|^{p\delta /n-1}<\infty\,,
\end{equation}
then $G$ supports the fractional $(p,p)$-Poincar\'e inequality 
\eqref{fractionalqp}.
\end{theorem}

\begin{proof}
We prove (1); the proof of (2) is similar. Let $\delta$ and $\tau$ in $(0,1)$ be given.
We use H\"older's inequality and Minkowski's inequality and then the Whitney decomposition 
to obtain
\begin{align}\label{integraltriangle}
\int_G \vert u(x)-u_G\vert^q\,dx
&\lesssim\int_{G}\vert u(x)-u_{Q_0^*}\vert^q\,dx\notag
\\&\le \sum_{Q\in W_{}}\int_{Q^*}
\vert u(x)-u_{Q_0^*}\vert ^q\,dx \notag\\
&\lesssim \sum_{Q\in W_{}}\int_{Q^*}\vert u(x)-u_{Q^*}\vert^q\,dx
+ \sum_{Q\in W_{}}\int_{Q^*}
\vert u_{Q^*}-u_{Q_0^*}\vert ^q\,dx\,.
\end{align}
Lemma \ref{inequality_cube} with $\rho=2\tau/3$ yields
\begin{align*}
&\int_{Q^*}
\vert u(x)-u_{ Q^*}\vert ^q\,dx
\\&\lesssim \vert Q^*\vert^{1+q(\delta /n -1/p)}\biggl(
\int_{ Q^*}\int_{Q^*\cap B^n(y,\,\rho \ell (Q^*))} 
\frac{\vert u(z)-u(y)\vert^p}{\vert z-y\vert^{n+\delta p}}\,dz\,dy
\biggr)^{q/p}\,.
\end{align*}
Inequalities \eqref{dist_est} and $(1+q\delta /n-q/p )(p/(p-q))>1$
imply
\begin{align*}
&\sum_{Q\in W_{}}\int_{ Q^*}\vert u(x)-u_{ Q^*}\vert^q\,dx\\
&\lesssim \sum_{Q \in W} |Q|^{
1+q\delta /n-q/p}
\biggl(\int_{Q^*} \int_{Q^*\cap B^n(y,\,\rho \ell (Q^*))} \frac{|u(y)-u(z)|^p}{|z-y|^{n+{\delta}p}}\,dz\,dy\bigg)^{q/p}
\\
&\le \bigg( \sum_{Q\in W} ( 
|Q|^{1+q\delta /n-q/p} )^{p/(p-q)} \bigg)^{(p-q)/p}\\
&\qquad\qquad\qquad
\bigg( \sum_{Q\in W} \int_{Q^*}\int_{Q^*\cap B^n(y,\,\rho \ell (Q^*))} \frac{|u(y)-u(z)|^p}{|z-y|^{n+\delta p}}\,dz\,dy \bigg)^{q/p}\\
&\lesssim
\bigg( \int_{G} \int_{G\cap B^n(y,\tau \dist(y,\partial G))} \frac{|u(y)-u(z)|^p}{|z-y|^{n+\delta p}}\,dz\,dy \bigg)^{q/p}.
\end{align*}

Next, we estimate the latter sum in (\ref{integraltriangle}).
By using chains from the chain decomposition we obtain
\begin{align*}
\sum_{Q\in W_{}}\int_{ Q^*}
\vert u_{ Q^*}-u_{ Q^*_0}\vert ^q\,dx
&\lesssim
\sum_{Q\in W}\vert Q\vert\,
\biggl(\sum_{j=1}^{k}
\vert u_{ Q^*_{j}}-
u_{ Q^*_{j-1}}
\vert \biggr)^q\\
&\le
\sum_{Q\in W}
{\ell(\C(Q^*))}^{q-1}\vert Q\vert\, \biggl(\sum_{j=1}^{k}
\vert u_{ Q^*_{j}}-
u_{ Q^*_{j-1}}
\vert^q\biggr)\,.
\end{align*}
Estimate $\max\{|Q^*_j|,|Q^*_{j-1}|\}\lesssim |Q^*_j\cap Q^*_{j-1}|$
and H\"older's inequality yield
\begin{align*}
\vert u_{ Q^*_{j}}-
u_{ Q^*_{j-1}}
\vert^q
&\lesssim
\sum_{i=j-1}^{j}
\biggl(\vert  Q^*_{i}\vert ^{-1}\int_{ Q^*_i}
\vert u(x)-u_{ Q^*_i}\vert\,dx\biggr)^q\\
&\le
\sum_{i=j-1}^{j}
\vert  Q^*_i\vert ^{-1}\int_{ Q^*_i}
\vert u(x)-u_{ Q^*_i}\vert^q\,dx\,.
\end{align*}
Lemma \ref{inequality_cube} with $\rho=2\tau/3$ implies
\begin{align*}
&\vert u_{ Q^*_{j}}-
u_{ Q^*_{j-1}}
\vert^q\\
&\lesssim
\sum_{i=j-1}^{j}
\vert Q^*_i\vert^{q(\delta /n -1/p)}
\biggl(
\int_{ Q^*_i}\int_{Q_i^*\cap B^n(y,\,\rho \ell (Q_i^*))}
\frac{\vert u(z)-u(y)\vert^p}{\vert z-y\vert^{n+\delta p}}\,dz\,dy\biggr)^{q/p}\,.
\end{align*}
We have obtained for the second sum in \eqref{integraltriangle}
\begin{align*}
&\sum_{Q\in W_{}}\int_{ Q^*}
\vert u_{ Q^*}-u_{ Q^*_0}\vert ^q\,dx\\
&\lesssim\sum_{Q\in W}
{\ell(\C(Q^*))}^{q-1}\vert Q\vert
\\&\qquad\qquad\biggl(
\sum_{j=0}^{k}
\vert  Q^*_j\vert ^{q(\delta /n-1/p)}
\biggl(\int_{ Q^*_j}\int_{Q_j^*\cap B^n(y,\,\rho \ell (Q_j^*)}
\frac{\vert u(z)-u(y)\vert^p}{\vert z-y\vert^{n+\delta p}}\,dz\,dy
\bigg)^{q/p}\biggr)\,.
\end{align*}
When we rearrange the double sum, we obtain
\begin{align*}
&\sum_{Q\in W_{}}\int_{ Q^*}
\vert u_{ Q^*}-u_{ Q^*_0}\vert ^q\,dx\\
&\lesssim \sum_{A\in W} \sum_{Q\in A(W)} 
\ell(\C(Q^*))^{q-1} |Q|\,|A|^{q(\delta /n-1/p)}\\
&\qquad\qquad\qquad
\bigg(\int_{A^*} \int_{A^*\cap B^n(y,\,\rho \ell (A^*))} \frac{|u(y)-u(z)|^p}{|z-y|^{n+\delta p}}\,dz\,dy\bigg)^{q/p}\,.
\end{align*}
H\"older's inequality with
$\left ( \frac{p}{q},\frac{p}{p-q}\right)$, and 
inequalities  (\ref{sharpe}) and \eqref{dist_est} yield
\begin{align*}
\sum_{Q\in W_{}}\int_{ Q^*}
\vert u_{ Q^*}-u_{ Q^*_0}\vert ^q\,dx
&\lesssim
\bigg(\sum_{A\in W} \int_{A^*} \int_{A^*\cap B^n(y,\,\rho \ell (A^*))} \frac{|u(y)-u(z)|^p}{|z-y|^{n+\delta p}}\,dz\,dy \bigg)^{q/p}\\
&\lesssim
\bigg(\int_{G}\int_{G\cap B^n(y,\tau\dist(y,\partial G))} \frac{|u(y)-u(z)|^p}{|z-y|^{n+\delta p}}\,dz\,dy \bigg)^{q/p}.
\end{align*}
Hence, $G$ supports 
the fractional $(q,p)$-Poincar\'e inequality \eqref{fractionalqp}.
\end{proof}

\begin{remark}
Let $G$ be a dounded domain in $\Rn$ and let $1\le p<\infty\,.$
By \cite[Theorem 6.6]{Hu} the estimate 
\begin{equation}\label{poincarecondition}
\sup_{A\in W}\sum_{Q\in A(W)}\ell(\mathcal{C}(Q^*))^{p-1}\,|Q|\,|A|^{p/n-1}
<\infty
\end{equation}
is a sufficient condition for the classical $(p,p)$-Poincar\'e inequality
to be valid in the domain $G$.
A comparison to our sufficient condition \eqref{pp} 
for the {\em fractional} $(p,p)$-Poincar\'e inequality
shows that
condition \eqref{poincarecondition} for the classical $(p,p)$-Poincar\'e inequality is weaker.
\end{remark}

\section{Positive results for $1$-John domains}\label{pos_1_John}

As an application of Theorem \ref{thmP} we show
that $1$-John domains support the fractional
$(p,p)$-Poincar\'e inequality, Theorem \ref{pp_John}. 
We also
consider fractional Sobolev--Poincar\'e inequalities,
Theorem \ref{fractional_sobolev_poincare} and Remark \ref{remaining_results}.
We recall that
bounded uniform and Lipschitz domains are examples of
$1$-John domains.


\begin{definition}\label{sjohn}
A bounded domain $G$ in $\R^n$, $n\ge 2$, is an {\em $s$-John domain}, $s\ge 1$, if
there is a point $x_0$ in $G$ and a constant $c>0$ such that every point $x$ in $G$ can be joined
to $x_0$ by a rectifiable curve $\gamma:[0,l]\to G$ parametrized by its arc length 
for which $\gamma(0)=x$, $\gamma(l)=x_0$, $l\le c$, and
\[
\dist(\gamma(t),\partial G)\ge t^s/c\quad \text{ for }t\in [0,l].
\]
The point $x_0$ is called an {\em $s$-John center} of $G$.
\end{definition}

If $G$ is a $1$-John domain, then its boundary $\partial G$ is porous in $\R^n$,
Definition \ref{porous}. The boundary of an $s$-John domain with $s>1$
may have positive Lebesgue $n$-measure, \cite{N1}, and thus it is not
necessarily porous in $\R^n$.

Let us construct a chain decomposition
of a given $s$-John domain $G$.
Let $Q\in W=W_G$ and fix a rectifiable curve
 $\gamma$ that is parametrized by its arc length
and joins the midpoints $x_Q$ and $x_0:=x_{Q_0}$, Definition \ref{sjohn}.
Assume that $x_{Q_0}$ lies in one of the
cubes intersecting $Q$. Join
$x_Q$ to $x_{Q_0}$ by an arc that is contained in $Q\cup Q_0$ and
whose length is comparable to $\ell(Q)$.
Otherwise there is $r>0$
such that $\gamma(r)$ lies
in the boundary of a Whitney cube $P$ that intersects $Q$ and
$\gamma(t)$ belongs to a cube
that is not intersecting $Q$ whenever $t\in (r,\ell(\gamma)]$. 
Join the midpoint
$x_Q$ to the midpoint $x_P$ by an
arc whose length is comparable to $\ell(Q)$ and
is in $Q\cup P$.
We iterate these steps
with $Q$ replaced by $P$, and we continue
until we reach $x_{Q_0}$. 
Let $\gamma_Q$ be this composed curve parametrized by its arc length.
It is straightforward 
to verify that $\ell(\gamma_Q)\le c$ and
\begin{equation}\label{esg}
t^s/c \le \,\dist(\gamma_Q(t),\partial G)\qquad \text{ if }\,t\in[0,\ell(\gamma_Q)],
\end{equation}
where $c>0$ depends
on  the $s$-John constant of $G$, $s$, and $n$.
Let $\mathcal{C}(Q^*)$ be a chain consisting
of cubes $A^*$ such that
$A\in W$ and $x_A\in \gamma_Q[0,\ell(\gamma_Q)]$.



For $1$-John domains we first have the following result.

\begin{theorem}\label{pp_John}
A $1$-John domain $G$ in $\R^n$
supports the fractional $(p,p)$-Poincar\'e inequality
\eqref{fractionalqp} if $1\le p<\infty$ and
$\tau,\delta\in(0,1)$.
\end{theorem}

\begin{proof}
We may assume
that $\diam(G)\le 1$.
By \eqref{esg} with $s=1$ and the fact that $\gamma_Q$, $Q\in W$, connects
the midpoints of cubes in $\mathcal{C}(Q^*)$, 
\begin{equation}\label{chain_length}
\ell(\mathcal{C}(Q^*))\le c\bigg(1+\log \frac{1}{\ell(Q)}\bigg),
\end{equation}
where the  constant $c$ is independent of $Q$.
If $A\in W$, then
\begin{equation}\label{contains}
\bigcup_{Q\in A(W)} Q\subset B^n(\omega_A,\min\{1,c\ell(A)\}),
\end{equation}
where 
$\omega_A$ is the closest point  in $\partial G$ to $x_A$ and
the constant $c>0$ is independent of $A$.
By \eqref{chain_length} and \eqref{contains} we obtain
\begin{align*}
&\sum_{Q\in A(W)}
\ell(\mathcal{C}(Q^*))^{p-1}|Q| 
\lesssim 
\sum_{Q\in A(W)}
|Q|\,\bigg(1+\log \frac{1}{\ell(Q)}\bigg)^{p-1}\\
&\lesssim \sum_{Q\in A(W)}
|Q|\,\bigg(1+\log^{p}\frac{1}{\ell(Q)}\bigg)\lesssim \sum_{Q\in A(W)}
\int_Q\bigg(1+\log^{p}\frac{1}{\dist(y,\partial G)}\bigg)\,dy\\
&\le \int_{B^n(\omega_A,\min\{1,c\ell(A)\})}  \bigg(1+\log^{p}\frac{1}{\dist(y,\partial G)}\bigg)\,dy.
\end{align*}
Since $\partial G$ is porous in $\R^n$, Lemma \ref{por_lem} yields
\[
\sum_{Q\in A(W)}
\ell(\mathcal{C}(Q^*))^{p-1}|Q| 
\lesssim 
|A|(1+\log^p \ell(A)^{-1})\lesssim |A|^{1-\delta p/n}.
\]
We have verified condition \eqref{pp} in Theorem \ref{thmP}. Hence,
the domain $G$ supports the fractional $(p,p)$-Poincar\'e inequality.
\end{proof}

We state an immediate corollary of Theorem \ref{pp_John}.

\begin{corollary}
Let $G$ be a bounded domain in $\R^n$, $n\ge 2$, and let $1\le p<\infty$,
$\delta,\tau \in (0,1)$.
Then $G$ supports the fractional $(p,p)$-Poincar\'e inequality \eqref{fractionalqp}
if $G$ is a uniform domain or a Lipschitz domain.
\end{corollary}

It is well known \cite[Theorem 5.1, Lemma 3.1]{B} that
$1$-John domains support Sobolev--Poincar\'e inequalities:
if $1\le p\le q\le np/(n-p)$, $p<n$, then there
is  $c>0$ such that, for
every $u\in W^{1,p}(G)$,
\begin{equation}\label{sobolev_poincare}
\bigg(\int_G |u(x)-u_G|^q\,dx\bigg)^{1/q}\le c\bigg(\int_G |\nabla u(x)|^p\,dx\bigg)^{1/p}.
\end{equation}
We consider
the corresponding fractional Sobolev--Poincar\'e inequalities
on $1$-John domains, Theorem \ref{fractional_sobolev_poincare}.
For the proof of this theorem we need the Riesz potentials
$I_\delta$, $\delta\in (0,n)$, that are defined for suitable $f$ by
\[
I_\delta(f)(x) = \int_{\R^n} \frac{f(y)}{|x-y|^{n-\delta}} dy.
\]
A proof of the following theorem is in \cite[Theorem 1]{He}.

\begin{theorem}\label{hedberg}
Let $0<\delta<n$, $1<p<q<\infty$, and $1/p-1/q=\delta/n$. Then
$||I_\delta(f)||_q \le c||f||_p$ for a constant $c>0$ independent
of $f\in L^p(\R^n)$.
\end{theorem}

We also need the following chaining lemma. It is
a slight modification
of \cite[Theorem 9.3]{HK}:
we add the new condition 3 but the proof
adapts to our setting, and
we omit the details.

\begin{lemma}\label{chaining_lemma}
Let $G$ in $\R^n$ be a $1$-John domain 
whose $1$-John constant is $c_J>1$.
Fix a number $M>1$. Denote by $x_0\in G$ the $1$-John center of $G$, and let
\[
B_0:=B(x_0,\mathrm{dist}(x_0,\partial G)/4Mc_J).
\]
Then, there is a constant $c>0$, depending on $G$, $M$, and $n$,
as follows:  given $x\in G$ there
is a sequence of balls 
$B_i = B(x_i,r_i)\subset G$, $i=0,1,\ldots$, such that for all $i=0,1,\ldots$, the
following conditions {\em 1--5} hold:
\begin{itemize}
\item[1.] $|B_i\cup B_{i+1}|\le c|B_i\cap B_{i+1}|$;
\item[2.] $\mathrm{dist}(x,B_i)\le cr_i$;
\item[3.] $\mathrm{dist}(B_i,\partial G)\ge Mr_i$;
\item [4.] $|x-x_i|\le cr_i$ and $r_j\to 0$ as $j\to\infty$;
\item [5.] $\sum_{j=0}^\infty \chi_{B_j}\le c\chi_G$.
\end{itemize}
\end{lemma}

The following result is a fractional Sobolev--Poincar\'e inequality
for $1$-John domains.

\begin{theorem}\label{fractional_sobolev_poincare}
Assume that $G$ is a $1$-John domain in $\R^n$, $n\ge 2$.
Suppose that $\tau,\delta\in (0,1)$, $p<n/\delta$, and
\[
1<p\le q\le \frac{np}{n-\delta p}.
\]
Then $G$ supports
the fractional $(q,p)$-Poincar\'e inequality \eqref{fractionalqp}.
\end{theorem}

\begin{proof}
By H\"older's inequality we may assume that $q=np/(n-\delta p)$.
Fix $\tau\in (0,1)$ and let $u\in L^p(G)$. 
Let  $x\in G$ be a  Lebesgue point of $u$, and consider the
associated balls $B_i=B(x_i,r_i)$ from Lemma \ref{chaining_lemma}
satisfying conditions 1--5 with $M>2/\tau$.

The following holds: for all $i$,
\begin{equation}\label{ball_incl}
B_i\subset B^n(y,\tau\dist (y,\partial G)),\qquad \text{ if }y\in B_i.
\end{equation}
Namely,
let us fix $y\in B_i$ and let $z$ be any point in $B_i\,$.
Then, by condition 3 in Lemma \ref{chaining_lemma},
\begin{align*}
\vert z-y\vert
&\le \vert y-x_i\vert+\vert x_i-z\vert
\le 2r_i\le 2\frac{\dist (B_i,\partial G)}{M}\\
&\le \frac{2}{M}\dist (y,\partial G) <\tau\dist (y,\partial G).
\end{align*}
By the Lebesgue differentiation theorem and condition 4 in Lemma \ref{chaining_lemma},
\begin{equation*}
u(x)=\lim_{i\to\infty}\frac{1}{\vert B_i\vert}\int_{B_i}u(y)\,dy=\lim_{i\to\infty} u_{B_i}\,.
\end{equation*}
Hence, by condition 1 in Lemma \ref{chaining_lemma}, we obtain
\begin{align*}
\vert u(x)-u_{B_0}\vert
&\le\sum_{i=0}^{\infty}\vert u_{B_i}-u_{B_{i+1}}\vert\\
&\le\sum_{i=0}^{\infty}\Bigl(\vert u_{B_i}-u_{B_i\cap B_{i+1}}\vert
+\vert u_{B_{i+1}}-u_{B_i\cap B_{i+1}}\vert\Bigr)\\
&\lesssim\sum_{i=0}^{\infty}
\frac{1}{|B_i|}\int_{B_i}\vert u(y)-u_{B_i}\vert\,dy\,.
\end{align*}
For a ball $B_i$,
\begin{equation}\label{ball_est}
\begin{split}
\frac{1}{\vert B_i\vert}\int_{B_i}\vert u(y)-u_{B_i}\vert\,dy
&=\frac{1}{\vert B_i\vert}\int_{B_i}
\bigg\vert 
\frac{1}{\vert B_i\vert}
\int_{B_i} ( u(y)-u(z))\,dz\bigg\vert\,dy\\
&\le \frac{1}{\vert B_i\vert}\int_{B_i}
\biggl(
\frac{1}{\vert B_i\vert}
\int_{B_i}\vert u(y)-u(z)\vert^p\,dz\biggr)^{1/p}\,dy\\
&=\frac{1}{\vert B_i\vert^{1+1/p}}
\int_{B_i}\biggl(
\int_{B_i}\vert u(y)-u(z)\vert^p\,dz\biggr)^{1/p}\,dy\\
&\lesssim
\vert B_i\vert^{\delta /n -1}\int_{B_i}\biggl(\int_{B_i}
\frac{\vert u(y)-u(z)\vert^p}{\vert y-z\vert^{n+\delta p}}\,dz\biggr)^{1/p}\,dy\,.
\end{split}
\end{equation}
Let us write
\begin{equation*}
g(y):=\biggl(\int_{G\cap B^n(y,\tau\dist(y,\partial G))}\frac{\vert u(y)-u(z)\vert ^p}{\vert y-z\vert ^{n+\delta p}}\,dz\biggr)^{1/p}\,.
\end{equation*}
By \eqref{ball_est}, \eqref{ball_incl} and condition 2 in Lemma \ref{chaining_lemma},
\begin{align*}
\sum_{i=0}^{\infty}\frac{1}{\vert B_i\vert}
\int_{B_i}\vert u(y)-u_{B_i}\vert \,dy
&\lesssim
\sum_{i=0}^{\infty}\vert B_i\vert^{\delta /n-1}\int_{B_i}
\biggl(\int_{B_i}\frac{\vert u(y)-u(z)\vert ^p}{\vert y-z\vert ^{n+\delta p}}\,dz\biggr)^{1/p}\,dy
\\
&\le
\sum_{i=0}^{\infty}\vert B_i\vert^{\delta /n-1}\int_{B_i}
\biggl(\int_{B^n(y,\tau\dist(y,\partial G))}\frac{\vert u(y)-u(z)\vert ^p}{\vert y-z\vert ^{n+\delta p}}\,dz\biggr)^{1/p}\,dy\\
&\lesssim
\sum_{i=0}^{\infty}r_i^{n(\delta /n-1)}\int_{B_i}g(y)\,dy\\
&\lesssim	
\sum_{i=0}^{\infty}\int_{B_i}\frac{g(y)}{\vert x-y\vert ^{n-\delta}}\,dy\,.
\end{align*}
By condition 5 in Lemma \ref{chaining_lemma},
\begin{equation}
\vert u(x)-u_{B_0}\vert\lesssim
\int_{G}\frac{g(y)}{\vert x-y\vert ^{n-\delta}}\,dy=I_{\delta}(\chi_G g)(x)
\end{equation}
for every Lebesgue point $x\in G$.
By integrating this inequality and using
Theorem \ref{hedberg}, we obtain
\begin{align*}
\biggl(\int_G \vert u(x)-u_{B_0}\vert ^{q}\,dx\biggr)^{1/q}
&\lesssim
\|I_\delta(\chi_G g)\|_q \lesssim \|\chi_G g\|_p\\
&=
\biggl(\int_{G}\int_{G\cap B^n(y,\tau\,\dist(y,\partial G))}\frac{\vert u(y)-u(z)\vert ^p}{\vert y-z\vert ^{n+\delta p}}\,dz\,dy\biggr)^{1/p}\,.
\end{align*}
Inequality \eqref{fractionalqp} follows.
\end{proof}

\begin{remark}\label{remaining_results}
The proof
of Theorem \ref{fractional_sobolev_poincare} 
also gives the following result:
Suppose that $G$ is a $1$-John domain in $\R^n$. Let
$\tau,\delta\in (0,1)$ and let $p,q\in [1,\infty)$ be such that
\[
0\le 1/p-1/q<\delta/n.
\]
Then $G$ supports
the fractional $(q,p)$-Poincar\'e inequality \eqref{fractionalqp}.
Indeed, it suffices to recall that the  linear operator $f\mapsto I_\delta(\chi_G f)$
is bounded from $L^p(G)$ to $L^q(G)$,
\cite[Lemma 7.12]{GT}.
\end{remark}

\section{Positive results for $s$-John domains with $s>1$}

We prove
the fractional $(1,p)$-Poincar\'e inequality
\eqref{fractionalqp} for $s$-John domains,
Theorem \ref{sharp}.  We show in  Section \ref{sharp_sec}
that this result is sharp in terms of the restriction on $p$,
Theorem \ref{1p_counter}.

\begin{theorem}\label{sharp} Let
$s>1$, $1<p<\infty$, $\lambda \in [n-1,n)$, and let $\delta,\tau\in (0,1)$. 
Suppose that
\begin{equation}\label{oletus}
s<\frac{n+1-\lambda}{1-\delta},\quad p>\frac{s(n-1)-\lambda+1}{n-s(1-\delta)-\lambda+1}.
\end{equation}
Let $G$ be an $s$-John domain in $\R^n$ 
such that $\mathrm{dim}_{\mathcal{M}}(\partial G)\le \lambda$.
Then $G$ supports the fractional $(1,p)$-Poincar\'e inequality \eqref{fractionalqp}.
\end{theorem}

We need  preparations for the proof of Theorem \ref{sharp}.

By scaling we may assume that $\diam(G)\le 1$. Hence,
the side lengths of all Whitney cubes in $W=W_G$ are bounded by one
and 
\begin{equation}\label{w_identity}
W=\bigcup_{j=0}^\infty W_j,
\end{equation}
where each
$W_j$ stands for the family of cubes $A\in W$ with
$\ell(A)=2^{-j}$.

For a given $s$-John domain $G$, we
consider its chain decomposition that is constructed in Section \ref{pos_1_John}.
Given $j,k\in\N$ and $\sigma\ge 1$ we define
\[
W_{j,k,\sigma}:=
\{A\in W_j\mid 2^{-(j-k)n}\le |\cup A(W)\,|\le \sigma\cdot 2^{-(j-k-1)n}\}.
\]
The following
lemma from \cite[Lemma 4.7]{HH-SV} gives the properties we need for
this chain decomposition of $G$.

The integer part of $\alpha\in\R$ is denoted by $[\alpha]$.

\begin{lemma}\label{sest}
Let $s>1$ and let $G$ be an $s$-John domain in $\R^n$ 
such that $\diam(G)\le 1$ and $\mathrm{dim}_{\mathcal{M}}(\partial G)< \lambda\in [n-1,n)$.
Then, there is a constant $\sigma\ge 1$ such that
\begin{equation}\label{kirjoitus}
 W_j= \bigcup_{k=0}^{[j-j/s]}  W_{j,k,\sigma}\qquad \text{ for every } j\in \N.
\end{equation}
Furthermore, if $k\in \{0,1,\ldots,[j-j/s]\}$,  then
\begin{equation}\label{sid}
\sharp W_{j,k,\sigma}\le c2^{-kn} 2^{j(n+1+(\lambda-n-1)/s)}.
\end{equation}
The positive constant $c$ depends on $n$, $s$, $\partial G$, and the $s$-John constant
of the domain $G$.
\end{lemma}

We are ready for the proof of Theorem \ref{sharp}.

\begin{proof}[{\it Proof of Theorem \ref{sharp}}]
Choose $\lambda'\in (\lambda,n)$ such that
\eqref{oletus} is true if $\lambda$ is replaced
by $\lambda'$. Then
$\mathrm{dim}_{\mathcal{M}}(\partial G)<\lambda'$ and hence we may assume
that  $\mathrm{dim}_{\mathcal{M}}(\partial G)$ is
strictly less than $\lambda\in [n-1,n)$.

By Theorem \ref{thmP}
it is enough to prove the finiteness of
\[
\Sigma := \sum_{A\in W} \bigg(\sum_{Q\in A({W})}
|Q|\, |A|^{\delta/n-1/p}\bigg)^{p/(p-1)}
=\sum_{A\in W}\big( |\cup A(W)\,|\, |A|^{\delta/n-1/p} \big)^{p/(p-1)},
\]
where the chain decomposition of $G$ is given by Lemma \ref{sest}.
By \eqref{w_identity} and \eqref{kirjoitus} in Lemma \ref{sest}
\begin{align*}
\Sigma &=\sum_{j=0}^\infty 
\sum_{k=0}^{[j-j/s]} 
\sum_{A\in W_{j,k,\sigma}} \big(  |\cup A(W)\,|\,  |A|^{\delta/n-1/p}\big)^{p/(p-1)}.
\end{align*}
Then, by using the definition of $ W_{j,k,\sigma}$ and \eqref{sid} 
from Lemma \ref{sest} we obtain
the estimate
\begin{align*}
\Sigma&\lesssim \sum_{j=0}^\infty \sum_{k=0}^{[j-j/s]} 2^{-kn} 2^{j(n+1+(\lambda-n-1)/s)}\cdot \big( 2^{-(j-k)n}\cdot 2^{-jn(\delta/n-1/p)}\big)^{p/(p-1)}\\
&= \sum_{j=0}^\infty \sum_{k=0}^{[j-j/s]}  2^{kn(p/(p-1)-1)} 2^{j(n+1+(\lambda-n-1)/s-np/(p-1)-\delta p/(p-1)+n/(p-1))}.
\end{align*}
Let us fix $j$ and $k$ as in the summation above. Then,
\[
kn\bigg(\frac{p}{p-1}-1\bigg) \le n(j-j/s)\bigg(\frac{p}{p-1}-1\bigg)
=\frac{jn(1-1/s)}{p-1}.
\]
The trivial estimate $[j-j/s]\le j$ implies that
\begin{align*}
\Sigma &\lesssim \sum_{j=0}^\infty j\cdot    2^{j(n(1-1/s)/(p-1)+n+1+(\lambda-n-1)/s-np/(p-1)-\delta p/(p-1)+n/(p-1))}\\
&=\sum_{j=0}^\infty j\cdot    2^{j(ns-s+\lambda p-\lambda-np-p+1-p(\delta-1)s)/s(p-1)}.
\end{align*}
By \eqref{oletus}
the last series converges.
\end{proof}

\section{Sharpness of Theorem \ref{sharp}}\label{sharp_sec}

We show
that Theorem \ref{sharp} is sharp
by proving Theorem \ref{1p_counter}. For this purpose we construct $s$-John domains which do not support 
the fractional $(1,p)$-Poincar\'e inequality \eqref{fractionalqp}
for certain values of $p$.

Let us recall the construction of the $s$-version of a given $1$-John domain $G$,
\cite{HH-SV}.
We
may assume that the diameter of $G$ is restricted by condition
\begin{equation}\label{w_restr}
(\ell(Q)/8)^s \le \ell(Q)/32,\qquad \text{ if }Q\in W_G.
\end{equation}
Let $Q$ be a closed cube in $\R^n$  centered at $x=(x_1,\ldots,x_n)$ and
whose side length $\ell=\ell(Q)$ satisfies 
$(\ell/8)^s\le \ell/32$. 
Thus,
$Q=\prod_{i=1}^n\,[x_i-\ell/2,x_i+\ell/2]$.
The {\em room} in $Q$ is the open cube
\[
R(Q):=\mathrm{int}(\frac{1}{4}Q)=\prod_{i=1}^{n} (x_i-\ell/8,x_i+\ell/8)
\]
centered at $x$ with side length $\ell/4$.
The {\em $s$-passage} in $Q$ is the open set
\[
P_s(Q):= \bigg(\prod_{i=1}^{n-1} \big(x_i-(\ell/8)^s,x_i+(\ell/8)^s \big)\bigg)  \times 
(x_n+\ell/8,x_n+\ell/4).
\]
Since
 $(\ell/8)^s < \ell/8$, we have $P_s(Q)\subset \frac{1}{2}Q$.
The {\em long $s$-passage} in $Q$ is the open set
\[
L_s(Q):= \bigg(\prod_{i=1}^{n-1} \big(x_i-(\ell/8)^s,x_i+(\ell/8)^s \big)\bigg)  \times 
(x_n,x_n+\ell/2)\subset Q.
\]
The {\em $s$-apartment} in $Q$ is the set
\begin{equation}\label{sap}
A_s(Q):= L_s(Q)\cup (Q\setminus (\partial R(Q)\cup \partial P_s(Q)))\subset Q.
\end{equation}

\begin{definition}\label{s_version}
Let $G$ in $\Rn$ be a $1$-John domain and let $s>1$ be 
a number such that \eqref{w_restr} holds.
Then,
the {\em s-version of $G$} is the domain
\[
G_s := Q_0\cup \bigcup_{\substack{Q\in W_G\\Q\not=Q_0}} A_s(Q).
\]
Here
$Q_0\in W_G$ is the cube containing the $1$-John center $x_0$ of $G$.
\end{definition}
%


We 
construct test functions. Let $Q\in W_G$ be fixed, and
define
the {\em tiny $s$-passage} in $Q$ to be the open set
\[
T_s(Q):=
 \bigg(\prod_{i=1}^{n-1} \big(x_i-(\ell/8)^s,x_i+(\ell/8)^s \big)\bigg)  \times 
(x_n+5\ell/32,x_n+7\ell/32).
\]
Then, we define a continuous function
\[u^{A_s(Q)}\colon G_s\to \R\]
which has linear decay along the $n^{\text{\tiny th}}$  variable in $T_s(Q)$
and 
is constant in both components of $P_s(Q)\setminus T_s(Q)$, and satisfies
\begin{equation}\label{values}
u^{A_s(Q)}(x) =\begin{cases}
\ell(Q)^{(\lambda-n)/q},\qquad &\text{ if }x\in R(Q);\\
0,\qquad &\text{ if } x\in G_s\setminus (R(Q)\cup P_s(Q)).
\end{cases}
\end{equation}
In the sense of distributions in $G_s$,
\begin{equation}\label{gradient}
\nabla u^{A_s(Q)} = (0,\ldots,0,-16\ell(Q)^{(\lambda-n)/q-1}\chi_{T_s(Q)})
\end{equation}
pointwise almost everywhere.

The  reason why we do not let $u^{A_s(Q)}$ have
linear decay along the whole $s$-passage $P_s(Q)$ is that
we need the following property.

\begin{remark}\label{loc_rem}
Let $Q\in W_G$. Suppose that
$x\in G_s$ and $y\in B^n(x,\dist(x,\partial G_s))$ are such that
\[|u^{A_s(Q)}(x)-u^{A_s(Q)}(y)|\not=0.\] Then
$x$ and $y$ both belong to $P_s(Q)$.
This fact follows from the
assumption \eqref{w_restr}.
\end{remark}

The following proposition is the main tool for
proving Theorem \ref{1p_counter}.

\begin{prop}\label{sharp_counter}
Let $G$ be a $1$-John domain in $\R^n$ and $s>1$
be
such that \eqref{w_restr} holds. Suppose that
 \[
 \limsup_{k\to \infty} 2^{-\lambda k}\cdot\sharp  W_k
>0,\qquad \text{ where }\lambda=\mathrm{dim}_{\mathcal{M}}(\partial G)\in [n-1,n).
 \] 
Let
 $\delta,\tau\in (0,1)$ and
$1\le q<p<\infty$ be such that
\begin{equation}\label{rels}
\frac{(p-q)(\lambda-n)}{pq} + \frac{(s-1)(n-1)}{p} \ge 1-s(1-\delta).
\end{equation}
 Then
the $s$-version of $G$  
is an $s$-John domain 
with $\mathrm{dim}_{\mathcal{M}}(\partial G_s)=\lambda$
and $G_s$ does not support
the fractional $(q,p)$-Poincar\'e inequality \eqref{fractionalqp}.
\end{prop}


\begin{proof}
The fact
\[\mathrm{dim}_{\mathcal{M}}(\partial G_s)=\mathrm{dim}_{\mathcal{M}}(\partial G)=\lambda\]
is from \cite[Proposition 5.11]{HH-SV}. 
By \cite[Proposition 5.16]{HH-SV}, the domain
$G_s$ is an $s$-John domain.
Hence, it remains
to prove the failure of the fractional Poincar\'e inequality.

Let us choose
$k_0\in\N$ such that
\begin{equation*}\label{number}
\limsup_{k\to \infty} 2^{-\lambda(k-k_0)}\cdot \sharp  W_k>2.
\end{equation*}
This allows us to 
choose indices 
$j(k)$, $k\in \N$, inductively such that
\[\max\{k_0,-\log_2\ell(Q_0)\}< j(1)<j(2)<\dotsb\]
and
$\sharp W_{j(k)} \ge 2\cdot 2^{\lambda (j(k)-k_0)}$ for every $k\in \N$.
Let us write $M_j:=2^{[\lambda(j-k_0)]}$, where $[\lambda(j-k_0)]$ means the integer part 
of $\lambda(j-k_0)$, and let us
choose cubes \[Q_{j(k)}^1,\ldots,Q_{j(k)}^{2M_{j(k)}}\in  W_{j(k)}\setminus \{Q_0\}.\]

For every $m\in\N$ we define
\[
v_m := \sum_{k=1}^m \bigg(\sum_{i=1}^{M_{j(k)}} u^{A_s(Q_{j(k)}^i)}-\sum_{i=M_{j(k)}+1}^{2M_{j(k)}} u^{A_s(Q_{j(k)}^i)}\bigg).
\]
Note that
$(v_m)_{G_s}=0$ and 
\begin{align*}
A_m:&=\bigg(\int_{G_s} |v_m - (v_m)_{G_s}|^q\bigg)^{1/q}
= \bigg(\sum_{k=1}^m\sum_{i=1}^{2M_{j(k)}} \int_{G_s}|u^{A_s(Q_{j(k)}^i)}|^q\bigg)^{1/q}
\\&\ge 
\Big(m\cdot2\cdot 2^{\lambda(j(k)-k_0)-1}\cdot 
2^{-j(k)(\lambda-n)}
\cdot
4^{-n}\cdot 2^{-j(k)n}\Big)^{1/q} = 
c_{n,q,\lambda,k_0}m^{1/q}.
\end{align*}

Next we estimate the right hand side of \eqref{fractionalqp} with $u=v_m$.
We write
\[
G_s(x):=B^n(x,\dist(x,\partial G_s))\subset G_s\qquad \text{ for } x\in G_s.
\]
Remark \ref{loc_rem} yields: if
$x\in G_s$ and $y\in G_s(x)$ are such that
$|v_m(x)-v_m(y)|\not=0$, then
$x,y\in P_s(Q)$ for some Whitney cube $Q\in W_G$. 
By using this we obtain
\begin{align*}
B_m:&=\bigg(\int_{G_s} \int_{G_s(x)} \frac{|v_m(x)-v_m(y)|^p}{|x-y|^{n+\delta p}}\,dy\,dx\bigg)^{1/p}\\
&=\bigg(\sum_{Q\in W_G}  \int_{Q\cap G_s}\int_{G_s(x)}\frac{|v_m(x)-v_m(y)|^p}{|x-y|^{n+\delta p}}\,dy\,dx\bigg)^{1/p}\\
&=\bigg(\sum_{Q\in W_G}  \int_{P_s(Q)}\int_{P_s(Q)\cap G_s(x)}\frac{|v_m(x)-v_m(y)|^p}{|x-y|^{n+\delta p}}\,dy\,dx\bigg)^{1/p}\\
&=\bigg(\sum_{k=1}^m\sum_{i=1}^{2M_{j(k)}}  \int_{P_s(Q_{j(k)}^i)}\int_{P_s(Q_{j(k)}^i)\cap G_s(x)}\frac{|u^{A_s(Q_{j(k)}^i)}(x)-u^{A_s(Q_{j(k)}^i)}(y)|^p}{|x-y|^{n+\delta p}}\,dy\,dx\bigg)^{1/p}.
\end{align*}

Let us fix a cube $R=Q_{j(k)}^i$,  where
$k\in \{1,\ldots,m\}$ and $i\in \{1,2,\ldots,2M_{j(k)}\}$.
By \eqref{gradient}
\[
|u^{A_s(R)}(x)-u^{A_s(R)}(y)|\le  16\ell(R)^{(\lambda-n)/q-1}|x-y|,\qquad x,y\in P_s(R).
\]
Hence, 
\begin{align*}
\mathcal{I}_R:=&\int_{P_s(R)}\int_{P_s(R)\cap G_s(x)}\frac{|u^{A_s(R)}(x)-u^{A_s(R)}(y)|^p}{|x-y|^{n+\delta p}}\,dy\,dx\\
&\lesssim \ell(R)^{p(\lambda-n)/q-p}\int_{P_s(R)}\int_{P_s(R)\cap G_s(x)} |x-y|^{-n+(1-\delta)p}\,dy\,dx.
\end{align*}
Note that $G_s(x)\subset B^n(x,\ell(R)^s)$ if $x\in P_s(R)$. Thus,
\[
\int_{P_s(R)\cap G_s(x)} |x-y|^{-n+(1-\delta)p}dy\le \int_{B^n(0,\ell(R)^s)} |y|^{-n+(1-\delta)p}dy
\lesssim \ell(R)^{s(1-\delta)p},
\]
and it follows that
\begin{align*}
\mathcal{I}_R&\lesssim \ell(R)^{p(\lambda-n)/q-p}|P_s(R)|\ell(R)^{s(1-\delta)p}\\&= \ell(R)^{p(\lambda-n)/q-p+s(n-1)+1+s(1-\delta)p}=2^{-j(k)(p(\lambda-n)/q-p+s(n-1)+1+s(1-\delta)p)}.
\end{align*}
These estimates and inequality \eqref{rels} yield
\begin{align*}
B_m&\lesssim \bigg(\sum_{k=1}^m
2^{\lambda j(k)}2^{-j(k)(p(\lambda-n)/q-p+s(n-1)+1+s(1-\delta)p)}\bigg)^{1/p}\lesssim m^{1/p}.
\end{align*}

By using the assumption $q<p$ we obtain
\[
\frac{A_m}{B_m} \ge c_{n,s,p,q,k_0,\lambda,\delta}m^{1/q-1/p}\xrightarrow{m\to \infty} \infty.
\]
Hence, the domain $G_s$ does not support the fractional $(q,p)$-Poincar\'e inequality \eqref{fractionalqp} for any $\tau\in (0,1)$.
\end{proof}

The following theorem shows the
sharpness of Theorem \ref{sharp}.

\begin{theorem}\label{1p_counter}
Let $s>1$, $p\in (1,\infty)$, $\lambda\in [n-1,n)$, and let $\delta,\tau\in (0,1)$.
Suppose that
\[
s<\frac{n+1-\lambda}{1-\delta},\qquad p\le \frac{s(n-1)-\lambda+1}{n-s(1-\delta)-\lambda+1}.
\]
Then, there is an $s$-John domain $G_s$ in $\R^n$
with the following properties:
$\mathrm{dim}_{\mathcal{M}}(\partial G_s) = \lambda$ and 
$G_s$ does not support 
the fractional $(1,p)$-Poincar\'e inequality \eqref{fractionalqp}.
\end{theorem}

\begin{proof}
By \cite[Proposition 5.2]{HH-SV} there
is a $1$-John domain $G$ in $\R^n$
such that
$\mathrm{dim}_{\mathcal{M}}(\partial G)=\lambda$ and
$\limsup_{k\to\infty} 2^{-\lambda k} \cdot \sharp W_k >0$.
By scaling  we may also assume that \eqref{w_restr} holds.
Hence, by Proposition \ref{sharp_counter} the $s$-version
$G_s$ has required properties.
\end{proof}

\bibliographystyle{amsalpha}

\end{document}